\documentclass[letterpaper, 10 pt, conference]{ieeeconf}

\IEEEoverridecommandlockouts
\usepackage{graphicx}          
\usepackage{times}   
\usepackage{gensymb}
\usepackage{amssymb}
\usepackage{amsmath}
\usepackage{mathtools}
\usepackage{bbm}
\usepackage{epsfig}
\usepackage[ruled]{algorithm}
\usepackage{algpseudocode}
\newcommand{\subparagraph}{}
\usepackage{titlesec}
\usepackage{enumerate}

\newtheorem{thm}{Theorem}
\newtheorem{lem}{Lemma}
\newtheorem{prop}{Proposition}
\newtheorem{cor}{Corollary}
\newtheorem{assum}{Assumption}
\newtheorem{defn}{Definition}
\newtheorem{rem}{Remark}

\def\defeq{\coloneqq}

\allowdisplaybreaks[4]

\title{\LARGE\bf On the convergence of stochastic MPC to terminal modes of operation} 
\author{Diego Mu\~{n}oz-Carpintero$^{1}$ \and Mark Cannon$^{2}$%
\thanks{$^{1}$Electrical Engineering Department, Universidad de Chile, Av.\ Tupper 2007, Santiago, Chile, \texttt{dimunoz@ing.uchile.cl}}%
\thanks{$^{2}$Department of Engineering Science, University of Oxford, UK OX1 3PJ, \texttt{mark.cannon@eng.ox.ac.uk}}}

\begin{document}

\maketitle

\begin{abstract}
The stability of stochastic Model Predictive Control (MPC) subject to additive disturbances is often demonstrated in the literature
by constructing Lyapunov-like inequalities that guarantee closed-loop performance bounds and boundedness of the state, but convergence to a terminal control law is typically not shown. In this work we use results on general state space Markov chains to find conditions that guarantee convergence of disturbed nonlinear systems to terminal modes of operation, so that they converge in probability to \emph{a priori} known terminal linear feedback laws and achieve time-average performance equal to that of the terminal control law. We discuss implications for the convergence of control laws in stochastic MPC formulations, in particular we prove convergence for two formulations of stochastic MPC.
\end{abstract}

\section{Introduction}

Stochastic Model Predictive Control accounts for the stochastic nature of disturbances using knowledge of their probability distributions or random samples generated by an oracle~\cite{kouvaritakisandcannon2015,mesbah2016}.
%
Performance is typically evaluated by considering the expectation of a quadratic cost function, which is minimized over open-loop (or closed-loop) predicted control sequences subject to probabilistic constraints with specified probabilities.  
%
The motivation for this formulation is to reduce the conservatism of worst-case formulations of Robust MPC, and to explicitly account for stochastic disturbances in the optimization of constrained performance~\cite{mayne2014}.

Stability analyses of stochastic MPC can be divided into cases in which the disturbances are additive or multiplicative. In the latter case, Lyapunov stability and (in the regulation case) convergence to the origin is often proved. However, in the former case additive disturbances generally prevent the state from converging to the origin. For this reason the question of convergence of the state is often either ignored, or else a simplified convergence analysis is performed, focussing on the derivation of Lyapunov-like functions. In \cite{cannonetal2009a,cannonetal2009b,kouvaritakisetal2010} this difficulty is tackled through a redefinition of the cost function; the resulting performance bounds imply that the time-average stage cost converges to that of a linear state feedback (e.g.\ the unconstrained certainty equivalent optimal control law). 
The work of \cite{chatterjeeandlygeros2015} performs a detailed analysis of the convergence and performance of stochastic MPC techniques, considering typical stability notions of Markov chains and exploring the implications for stability and performance of the resulting Lyapunov-like inequalities (or geometric drift conditions).
In particular, certain Lyapunov-like inequalities provide the robust notion of input-to-state stability (ISS), and these inequalities may be used to obtain ultimate boundedness conditions on the system state. Thus, for example, the results of \cite{chatterjeeandlygeros2015} were used to conclude the stability of the stochastic MPC formulations in~\cite{paulsonetal2015,mishraetal2016}.

A shortcoming of the stability analyses so far discussed is that asymptotic convergence properties may be stronger than what is directly implied by Lyapunov-like inequalities. Recently \cite{lorenzenetal2017} presented a stochastic MPC strategy with additive disturbances, and included a convergence analysis showing that the asymptotic limit of the time-average performance is bounded by that of a terminal control law, and that the MPC law converges to the terminal law. These results are not directly obtained from the Lyapunov inequality. Instead, they come from particular properties of the MPC formulation, namely that if a candidate solution remains feasible for a given number of time steps, the system will  converge to a terminal set on which the terminal control law is feasible, and the probability of this occurring tends asymptotically to~$1$.

In this paper we perform an analysis similar to that of \cite{lorenzenetal2017}, but in a more general setting and using different tools. 
We show that, for a class of nonlinear systems, the existence of a Lyapunov function demonstrating input-to-state stability (ISS) ensures convergence of the state to a robust invariant set under the assumption that arbitrarily small disturbances have a non-vanishing probability. 
This result applies, for example, to the closed-loop system formed by a linear system subject to stochastic additive disturbances under an input-to-state stabilizing stochastic MPC law. 
Given a suitably chosen quadratic cost, the implied MPC law reduces to a linear optimal controller on this terminal set. 
Subsequently, using results on the convergence of Markov chains, we show convergence of the probability distribution of the state to a terminal mode. The paper illustrates these techniques by deriving new results on the convergence of the stochastic MPC laws proposed in~\cite{goulartandkerrigan2008} and~\cite{kouvaritakisetal2013}.

\section{Mathematical preliminaries}\label{sec:prelims}



This section introduces basic notation, definitions and a summary of relevant results on  Markov Chain convergence. 


The sets of non-negative integers and non-negative reals are denoted by ${\mathbb N}$ and ${\mathbb R}_+$, while ${\mathbb N}_k$ denotes $\{0,1,\ldots,k\}$. 
For a sequence ${\bf x}=\{x_k: k\in{\mathbb N} \}$, $x_{j|k}$ is the prediction of $x_{k+j}$ made at time $k$. 
The Minkowski sum of $X,Y\subseteq{\mathbb R}^n$ is $X\oplus Y \defeq\{x+y:x\in X,\,y\in Y \}$ and the Minkowski sum of a sequence of sets $\{X_j \}_{j\in\mathbb{N}_k}$ is denoted $\bigoplus_{j=0}^k X_j$. 
For $X\subseteq{\mathbb R}^n$, ${\bf 1}_X(x)$ is the indicator function of $X$. A continuous function $\phi:{\mathbb R}_+ \rightarrow {\mathbb R}_+$ is a ${\mathcal K}$-function if it is continuous, strictly increasing and $\phi(0)=0$, and it is a ${\mathcal K}_\infty$-function if it is a ${\mathcal K}$-function and $\phi(s)\rightarrow \infty$ as $s\rightarrow \infty$. A continuous function $\phi:{\mathbb R}_+\times{\mathbb R}_+\rightarrow {\mathbb R}_+$ is a  ${\mathcal K} {\mathcal L}$-function if for all $t\in{\mathbb R}_+$, $\phi(\cdot,t)$ is a ${\mathcal K}$-function and for all $s\in{\mathbb R}_+$, and $\phi(s,\cdot)$ is decreasing with $\phi(s,t)\rightarrow 0$ as $t\rightarrow \infty$.

A measurable space is a pair $({\bf X},\,{\mathcal B}({\bf X}))$, where ${\bf X}$ is an arbitrary set and ${\mathcal B}({\bf X})$ is a $\sigma$-algebra of ${\bf X}$. A measure on the space $({\bf X},\,{\mathcal B}({\bf X}))$ is a function $\mu:{\bf X}\rightarrow{\mathbb R}\cup\{\infty\}$ which is countably additive. The measure $\mu(\cdot)$ is said to be positive if $\mu(X)\ge0$ for all $X\in{\mathcal B}({\bf X})$, it is finite if $\mu({\bf X})<\infty$, and it is a probability measure if it is positive and $\mu({\bf X})=1$. Given two measurable spaces $({\bf X}_1,{\mathcal B}({\bf X}_1))$ and $({\bf X}_2,{\mathcal B}({\bf X}_2))$, a function $f:{\bf X}_1\rightarrow{\bf X}_2$ is said to be measurable if $f^{-1}\{X\}=\{x:f(x)\in X \}\in{\mathcal B}({\bf X}_1)$ for all sets $X\in{\mathcal B}({\bf X}_2)$. The support of the measure $\mu:{\bf X}\rightarrow{\mathbb R}_+\cup\{\infty\}$ is denoted ${\rm supp}(\mu)$ and is the closure of the set $\{X\in{\mathcal B}({\bf X}):\mu(X)>0\}$.
%
A probability space is defined by  the triple $(\Omega,{\mathcal F},{\mathbb P})$, 
where $\Omega$ is the sample space, ${\mathcal F}$ is a $\sigma$-algebra of $\Omega$ and ${\mathbb P}$ is a probability measure on $(\Omega,{\mathcal F})$.
Given a probability space $(\Omega,{\mathcal F},{\mathbb P})$ and measurable space $({\bf X},{\mathcal B}({\bf X}))$, a random variable $x$ on $\Omega$ is a measurable function $x:{\Omega}\rightarrow {\bf X}$.
The expected value of $x$ is ${\mathbb E}\{x\}$.


\subsection{Convergence of Markov chains}

\begin{defn}[Markov chains]
  Consider a measurable space $({\bf X},{\mathcal B}(\bf X))$
  and a stochastic process ${\bf x}\defeq \{x_k\in{\bf X}_k:$ $k\in \mathbb{N} \}$ defined on $(\Omega,{\mathcal F})$, where
  ${\mathcal F}$ is a $\sigma$-algebra on 
$\Omega\defeq\Omega_\infty$ where $\Omega_k= {\bf X}_0\times \cdots \times {\bf X}_k$ and
  ${\bf X}_j={\bf X},\, \forall j\in\mathbb{N}$.
  Then ${\bf x}$ is a time-homogenous Markov chain with transition probability function $P(x,{\mathcal A})\defeq {\mathbb P}\{x_{k+1}\in {\mathcal A} : x_k=x\}$ if the distribution of ${\bf x}$ satisfies the Markov property
\[
{\mathbb P}(x_{k+1}\in {\mathcal A} : x_j=\bar{x}_j,\,j\in\mathbb{N}_k)=P(\bar{x}_{k},{\mathcal A}).
\]
\end{defn}
\vspace{\baselineskip}

\begin{defn}[Invariant measure]
For the Markov chain ${\bf x}$ an invariant probability measure is a stationary distribution, i.e.\ a probability measure $\pi$ satisfying
\begin{equation}\label{invmeas}
\pi({\mathcal A})=\int{\pi(dx)P(x,{\mathcal A})}, \quad \forall {\mathcal A}\in {\mathcal B}(\bf X).
\end{equation}
\end{defn}

\begin{defn}[$\varphi$-irreducibility]
  The Markov chain ${\bf x}$ is said to be $\varphi$-irreducible if there exists a measure $\varphi$ on ${\mathcal B}({\bf X})$ such that if $\varphi({\mathcal A})>0$ then $L(x,{\mathcal A})>0$ for all $x\in{\bf X}$, where $L(x,{\mathcal A})$ is the probability that $x_k\in{\mathcal A}$ for some $k\in {\mathbb N}$ given that $x_0=x$.
\end{defn}

\begin{defn}[$d$-cycle]
A $d$-cycle is a collection of disjoint sets $\{D_i: i\in \mathbb{N}_{d-1}\}$ such that for $x_0\in D_i$, $P(x_{k},D_{{\rm mod}(k+i,d)})=1$ for all $k\in\mathbb{N}$. 
\end{defn}

\begin{defn}[Periodicity]
The period of ${\bf x}$ is the largest $d$ for which a $d$-cycle occurs. If $d=1$, ${\bf x}$ is aperiodic.
\end{defn}

We will make use of the following result to prove that a Markov chain is aperiodic \cite[Sec.~5.4.3]{meynandtweedie2009}.

\begin{thm}\label{aperiodicthm}
Let ${\bf x}$ be a $\varphi$-irreducible Markov chain with transition probability $P(x,{\mathcal A})\defeq{\mathbb P}\{x_{k+1}\in {\mathcal A} : x_k=x\}$. Assume that there is a set $C\in {\mathcal B}({\bf X})$ and a non-trivial measure $\nu(\cdot)$ on ${\mathcal B}({\bf X})$,  such that for all $x\in C$, $B\in{\mathcal B}(X)$
\[
P(x,B)\ge \nu(B) \text{ and } \nu(C)>0 .
\]
Then ${\bf x}$ is aperiodic.
\end{thm}

The main result on convergence of Markov chains to be used later is presented next.

\begin{thm}\label{markovconv}
  Let ${\bf x}$ be a $\varphi$-irreducible Markov chain with state space ${\bf X}\subseteq\mathbb{R}^{n_x}$
  such that
\begin{itemize}
\item[(i)]
${\bf x}$ is generated by the system $x_{k+1}=f(x_k,w_k)$, for continuous $f:{\bf X}\times{\mathbb W}^{n_w}\rightarrow{\bf X}$, where ${\bf w}\defeq\{w_k: k\in\mathbb{N}\}$ is a stochastic disturbance sequence in ${\mathbb R}^{n_w}$,
\item[(ii)] ${\bf x}$ is aperiodic,
\item[(iii)] ${\rm supp}(\varphi)$ has non-empty interior,
\item[(iv)] there is a measurable function $V:{\bf X}\rightarrow[0,\infty]$
  such that for any $c<\infty$ the sublevel set $C_V(c)\defeq \{y: V(y)\le c \}$ is compact, and there exists a compact set $C$ and a scalar $b$ satisfying for all $x_k\in {\bf X}$:
\begin{equation}\label{LyapMC}
\Delta V(x_k)\defeq  {\mathbb E}\{V(x_{k+1})\}-V(x_k)\le -1+b{\bf 1}_C(x_k) .
\end{equation}
%
\end{itemize}
Then an invariant probability measure $\pi(\cdot)$ exists satisfying 
\begin{equation}\label{measconv}
\lim_{k\rightarrow\infty}\sup_{{\mathcal A}\in{\mathcal B}(X)} \lvert P^k(x,{\mathcal A})-\pi({\mathcal A})\rvert = 0,
\end{equation}
where $P^k(x,{\mathcal A}) \defeq {\mathbb P}\{x_{k}\in {\mathcal A} : x_0=x\}$, and the Law of Large Numbers (LLN),
\begin{equation}\label{LLN}
\lim_{k\rightarrow\infty}\frac{1}{k}\sum_{j=1}^kh(x_j)\stackrel{a.s.}{=} {\mathbb E}_\pi\{h(x)\} ,
\end{equation}
holds for every $h:{\bf X}\rightarrow {\mathbb R}$ such that ${\mathbb E}_\pi(\lvert h(x) \rvert )<\infty$, where ${\mathbb E}_\pi\bigl(h(x)\bigr)\defeq\int{\pi(dx)h(x)}$.
\end{thm}

Theorems~\ref{aperiodicthm} and~\ref{markovconv} are rearrangements of results given in~\cite{meynandtweedie2009}, which we state here without proof.

\section{Problem Setting}\label{sec:setting}

Consider a discrete time nonlinear system given by
\begin{equation} \label{eq1} 
x_{k+1} =f(x_k,w_k)\defeq g(x_k)+Dw_k,
\end{equation} 
where $x_{k}\in {\bf X}\subseteq {\mathbb R}^{n_{x} } $ is the state, $w_{k}\in{\bf W}\subseteq{\mathbb R}^{n_w}$ is a disturbance and $g:{\bf X}\rightarrow {\bf X}$ is a continuous function such that $g(0)=0$. Current and future values of $w_k$ are unknown. 

\begin{assum}\label{distass}
$\{w_k\in{\bf W} : k\in {\mathbb N}\}$ is a zero mean, independent and identically distributed (i.i.d.) disturbance sequence with a non-singular probability distribution supported in ${\bf W}$, which is a bounded set containing the origin in its interior. Furthermore ${\mathbb P}\{||w||\le \lambda \}>0$ for all $\lambda>0$.
\hfill$\blacktriangleleft$
\end{assum}

\begin{rem}
The i.i.d.\ property as well as the null expected value are standard requirements in stochastic MPC formulations. The requirement that the disturbance is contained in a bounded set is necessary to guarantee that there is a robustly invariant set (as will be introduced later), and is also common in stochastic MPC \cite{kouvaritakisandcannon2015,mesbah2016}. Finally, the property that ${\mathbb P}\{||w||\le \lambda \}>0$ for any $\lambda>0$ is needed to prove the desired result. Although this assumption appears strong, it does not require continuity of the density function, and is satisfied by many common probability distributions.
\end{rem}

We consider in this paper nonlinear systems~\eqref{eq1} that have a linear terminal mode of operation to which one desires to prove convergence. This is typically the case for linear systems under MPC.
Thus the system dynamics are linear on some robustly invariant set, and state trajectories converge to a regime defined by an invariant probability measure within this set. This property is formally characterized next. 

\begin{assum}\label{terminalass}
There exists a bounded set ${\bf X}_f\subseteq{\bf X}$ containing the origin in its interior, such that for all $x\in {\bf X}_f$:
\begin{itemize}
  \item[(i)] $f(x,w)\in{\bf X}_f$ for all $w\in{\bf W} $, i.e. ${\bf X}_f$ is a robustly invariant set.
\item[(ii)] $f(x,w)=Ax+Dw$ for all $x\in{\bf X}_f$, where $A$ is Schur stable and the pair $(A,D)$ is controllable. \hfill$\blacktriangleleft$
\end{itemize}
\end{assum}

The linear system 
\begin{equation}\label{lineardynamics}
x_{k+1}=Ax_k+Dw_k
\end{equation}
where $A$ is Schur stable and $(A,D)$ is controllable, defines an invariant probability measure $\pi(\cdot)$ and transition probability function $P(x,\cdot)$ that satisfy \eqref{invmeas}. Furthermore, the system converges to this stationary regime \cite{meynandtweedie2009}. In this case $\pi(\cdot)$ is the probability measure of $\sum_{k=0}^\infty A^kDw_k$, which is supported in the minimal invariant set ${\bf X}_\infty=\bigoplus_{k=0}^\infty A^kD{\bf W}$.

We make use of the concept of input-to-state stability.

\begin{defn}[Input-to-state stability (ISS)~\cite{jiangandwang2001}] For the system \eqref{eq1}, the origin is input-to-state stable (ISS) with region of attraction ${\bf X}\subseteq{\mathbb R}^{n_x}$ if there exist a ${\mathcal K}{\mathcal L}$-function $\beta(\cdot,\cdot)$ and a ${\mathcal K}$-function $\gamma(\cdot)$ such that, for all $x_0\in{\bf X}$ and $w \in {\bf W}$, the system \eqref{eq1} satisfies $x_k\in{\bf X}$ for all $k\in{\mathbb N}$ and
\begin{equation} \label{iss}
||x_k||\le \beta(||x_0||,k)+\gamma(\sup\{||w_t||:t\in \mathbb{N}_{k-1} \} ).
\end{equation}
\end{defn}

The following lemmas provide a useful sufficient condition that is easy to check for guaranteeing ISS.

\begin{lem}[\cite{jiangandwang2001}]\label{lemiss}
For system \eqref{eq1}, the origin is ISS with region of attraction ${\bf X}\subseteq{\mathbb R}^{n_x}$ if ${\bf X}$ contains the origin in its interior and is robustly invariant, and there exist a continuous function $V:{\bf X}\rightarrow {\mathbb R}_+$ (called an ISS-Lyapunov function), ${\mathcal K}_\infty$-functions $\alpha_1$, $\alpha_2$ and $\alpha_3$, and a ${\mathcal K}$-function $\sigma$ such that
\begin{subequations}
\begin{gather}
  \alpha_1(||x||)\le V(x) \le \alpha_2(||x||), 
  \label{lyapkappa}\\
  V\left(f(x,w)\right)-V(x)\le -\alpha_3(||x||)+\sigma(||w||), 
  \label{isslyap}
  \end{gather}
\end{subequations}
for all $x\in{\bf X}$ and $w\in{\bf W}$.
\end{lem}

\begin{lem}[\cite{goulartetal2006}]\label{lemiss2}
Let ${\bf X}\subseteq {\mathbb R}^{n_x}$ be a robust positively invariant set for \eqref{eq1} containing the origin in its interior, and let ${\mathcal K}_\infty$-functions $\alpha_1$, $\alpha_2$, $\alpha_3$ and a Lipschitz continuous function $V:{\bf X}\rightarrow {\mathbb R}_+$ exist such that, for all $x\in{\bf X}$,
\begin{subequations}
\begin{gather}
  \alpha_1(||x||)\le V(x) \le \alpha_2(||x||),
  \label{lyapkappa2}\\
  V\left(f(x,0)\right)-V(x)\le -\alpha_3(||x||).
  \label{isslyap2}
\end{gather}
\end{subequations}
Then $V(\cdot)$ is an ISS-Lyanpunov function and the origin is ISS for system~\eqref{eq1} with region of attraction ${\bf X}$ if: (i)~$f(\cdot,\cdot)$ is Lipschitz continuous on ${\bf X}\times{\bf W}$; or (ii)~$f(x,w):=g(x)+w$, where $g:{\bf X}\rightarrow{\mathbb R}^{n_x}$ is continuous on ${\bf X}$.
\end{lem}

\begin{assum}\label{lyapunov}
The system \eqref{eq1} has an ISS-Lyapunov function satisfying the conditions of either Lemma \ref{lemiss} or \ref{lemiss2}.
\hfill$\blacktriangleleft$
\end{assum}

Note that, if the conditions of Lemma \ref{lemiss2} hold, then the conditions of Lemma \ref{lemiss} necessarily also hold \cite{goulartetal2006}.
Furthermore, ISS implies~\cite{jiangandwang2001} that: (i) the origin is asymptotically stable for $x_{k+1}=f(x_k,0)$; (ii)  all state trajectories are bounded for bounded $w(\cdot)$; and (iii) 
all trajectories converge to the origin as $k\rightarrow\infty$  if $w_k\rightarrow0$. 

These properties do not directly guarantee convergence to the terminal mode of operation defined in Assumption \ref{terminalass}. In this work, however, we will assume that the system possesses an ISS-Lyapunov function, which implies that the origin is ISS. This property is coupled with the stochastic nature of the disturbance sequence to prove convergence to ${\bf X}_f$.

\section{Main results}\label{sec:results}

This section shows that, under Assumptions \ref{distass}, \ref{terminalass} and \ref{lyapunov}, the system converges to the terminal regime described by Assumption~\ref{terminalass}.

\begin{thm}
  Under Assumptions \ref{distass}, \ref{terminalass} and \ref{lyapunov}, the Markov chain ${\bf x}=\{x_k\in{\bf X} : k \in {\mathbb N}\}$ defined by \eqref{eq1} is $\varphi$-irreducible.
\end{thm}

\begin{proof}
First note that, since $(A,D)$ is controllable and $w_k$ has a non-singular distribution, system \eqref{lineardynamics} is necessarily $\varphi$-irreducible \cite[Prop.~6.3.5]{meynandtweedie2009}. 

Secondly, from \eqref{iss} it follows that the state of (\ref{eq1}) converges to the origin if the disturbance sequence satisfies $w_k=0$ for all $k\in{\mathbb N}$. Since $\gamma(\cdot)$ is continuous, it follows that for any neighbourhood ${\mathcal E}$ of the origin, there is an $\epsilon>0$ such that ${\mathcal E}$ will be reached in $m$ steps whenever $||w_k||\le \epsilon$ for all $k\in\mathbb{N}_{m-1}$. Thus, there is a non-zero probability that ${\mathcal E}$ will be reached. In particular, let ${\mathcal E}={\bf X}_f$. Then, for any initial state, there is a non-zero probability of reaching ${\bf X}_f$, and, since ${\bf X}_f$ is robustly invariant, of remaining in this set at all subsequent times. Moreover, 
since \eqref{lineardynamics} is $\varphi$-irreducible, there is a non-zero probability of reaching any set ${\mathcal A}\in{\mathcal B}({\bf X})$ with $\varphi({\mathcal A})>0$. Thus, ${\bf x}$ is $\varphi$-irreducible with the invariant measure of system \eqref{lineardynamics}, namely with $\varphi=\pi$. \end{proof}

\begin{prop}\label{prop1} Under Assumptions \ref{distass}, \ref{terminalass} and \ref{lyapunov} we have:
  \begin{itemize}
\item[(i)]
  ${\rm supp}(\varphi)$ has a non-empty interior,
\item[(ii)]
  There is a measurable function $V:{\bf X}\rightarrow[0,\infty]$
  such that for any $c<\infty$ the sublevel set $C_V(c)=\{y: V(y)\le c \}$ is compact, and there is a compact set $C$ satisfying \eqref{LyapMC} for all $x\in {\bf X}$.
\end{itemize}
\end{prop}
\begin{proof}
  (i) The support of $\varphi(\cdot)$ is ${\bf X}_\infty=\bigoplus_{k=0}^\infty A^kD{\bf W}$, which clearly contains $D{\bf W}$, so the interior of ${\rm supp}(\varphi)$ cannot be empty since $0\in{\rm int}(\bf W)$ and $(A,D)$ is controllable.

(ii) Let $C_{\alpha_3}(d')=\{x:\alpha_3(x)\le d' \}$, with $d'>d$, where $d=\max\left({\mathbb E}\{\sigma(||w||)\}:w\in {\bf W}\right)$. Both $C_V(c)$, for any $c>0$, and $C_{\alpha_3}(d')$ are compact due to Assumption \ref{lyapunov}, particularly from continuity of $V(\cdot)$ and $\alpha_3(\cdot)$, and \eqref{lyapkappa}. Condition \eqref{isslyap} implies that $\Delta V(x)\le-\alpha_3(x)+d\le-(d'-d)$ for $x\notin C_{\alpha_3}(d')$ and trivially $\Delta V(x)\le-\alpha_3(x)+d\le d$ for $x\in C_{\alpha_3}(d')$. This can be rewritten as $\Delta \tilde{V}(x)\le-1+b{1}_{C_{\alpha_3}(d')}(x)$ with $\tilde{V}(x)=V(x)/(d'-d)$ and $b=1+d/(d'-d)$. Hence (\ref{LyapMC}) holds with $C=C_{\alpha_3}(d')$ and $V(x)$ replaced by $\tilde{V}(x)$, and $C_{\tilde{V}}(c)$ is compact for any $c>0$.
\end{proof}

\begin{lem}\label{lem3}
Under Assumptions \ref{distass}, \ref{terminalass} and \ref{lyapunov}, the Markov chain ${\bf x}$ is aperiodic.
\end{lem}
\begin{proof}
Let $C = \{x\in{\bf X} : \|x\| \leq r\}$ where the radius $r>0$ is small enough that the set $\Omega=\cap_{x\in C}\bigl(D{\bf W}\oplus \{g(x)\}\bigr)$ is non-empty (existence of sufficiently small $r$ is ensured by the assumption that $g(\cdot)$ is continuous with $g(0)=0$). The transition probability function then satisfies, for all $x\in C$,
\begin{align*}
P(x,B)&={\mathbb P}\bigl(g(x)+Dw\in B\bigr)\\
      &=\int_{D{\bf W}} 1_{B}(g(x)+\tilde{w}) \gamma(d\tilde{w}) \\
&=\int_{D{\bf W}\oplus \{g(x)\}} 1_{B}(\tilde{w}) \gamma(d\tilde{w}) \\
& \ge \int_{\Omega} 1_{B}(\tilde{w}) \gamma(d\tilde{w}) \triangleq \mu_{\tilde{W}}(B)
\end{align*}
where $\tilde{w}=Dw$ and $\gamma(\cdot)$ is the probability distribution of $\tilde{w}$. Thus, $C$ satisfies the conditions of Theorem \ref{aperiodicthm} with $\nu=\mu_{\tilde{W}}$, thus proving that ${\bf x}$ is an aperiodic Markov chain.
\end{proof}

Proposition~\ref{prop1} and Lemma~\ref{lem3} allow us to invoke Theorem~\ref{markovconv}, which directly implies the following theorem.

\begin{thm}\label{mainthmgeneral}
Under Assumptions \ref{distass}, \ref{terminalass} and \ref{lyapunov}, system \eqref{eq1} satisfies \eqref{measconv} and \eqref{LLN}.
\end{thm}

As will be seen in the upcoming sections of the paper, it is relevant to analyse the particular case when the dynamics in the terminal mode are linear. The following corollary presents a direct consequence of Theorem \ref{mainthmgeneral}.

\begin{cor}\label{maincor}
  Under Assumptions \ref{distass}, \ref{terminalass} and \ref{lyapunov},
  system~\eqref{eq1} satisfies%
\[
\lim_{k\rightarrow\infty}{\mathbb P}\{ x_k\in {\bf X}_\infty \}=1,
\qquad 
{\bf X}_\infty=\smash[t]{\bigoplus_{k=0}^\infty} A^kD{\bf W}.
\]
\end{cor}
\vspace{2mm}

\begin{proof}
Equation \eqref{measconv} is satisfied due to Theorem \ref{mainthmgeneral}. Since $f(x,w)=Ax+Dw$ for $x\in {\bf X}_f$, the unique invariant measure $\pi$ is the measure of $X_\infty=\sum_{k=0}^\infty A^kDw_k$. By choosing ${\mathcal A}={\bf X}_\infty$, where ${\bf X}_\infty$ is the support of $X_\infty$ and clearly $\pi({\bf X}_\infty)=1$, it follows that $\lim_{k\to\infty}P^k(x,{\bf X}_\infty)=1$ which is the desired result.
\end{proof}

Thus, under Assumptions \ref{distass}, \ref{terminalass} and \ref{lyapunov}, it has been demonstrated that for the nonlinear system \eqref{eq1}, the probability distribution of the state $x_k$ and the time-average value of $h(x_k)$ in \eqref{LLN} converge to that of the terminal mode. This implies that, under the assumption that the terminal mode dynamics are linear, the nonlinear system converges to the minimal invariant set of the linear terminal mode dynamics.

\section{Implications for stability and convergence of stochastic MPC to a terminal control law}\label{sec:applications}

In this section we use the preceding results to obtain convergence results for two MPC formulations. The first of these employs a control policy with affine dependence on the disturbance input~\cite{goulartetal2006,goulartandkerrigan2008}. Convergence to a minimal invariant set for a variant of this formulation was proved in~\cite{wangetal2008} by redefining the MPC cost function and control policy. The second formulation also uses a control policy that is affine in the disturbances, but the feedback law has a different structure (striped, extending over an infinite prediction horizon) with gains computed offline \cite{kouvaritakisetal2013}. Convergence for this formulation is shown here for the first time.

For both strategies we consider a system given by 
\begin{equation}\label{lindyn}
x_{k+1}=Ax_k+Bu_k+Dw_k
\end{equation}
where $A,B,D$ are matrices  of conformal dimensions, and a measurement of the current state, $x_k$, is available at time $k$. It is assumed that $(A,B)$ is stabilizable. Also, the disturbance sequence ${\bf w}\defeq\{w_k\in{\bf W}: k\in {\mathbb N}\}$ is i.i.d., ${\mathbb E}\{w_k \}=0$, and the probability distribution of $w_k$ is finitely supported in ${\bf W}$, a bounded set that contains the origin in its interior. These assumptions are part of the setting of \cite{goulartetal2006,goulartandkerrigan2008,kouvaritakisetal2013}. Here we additionally assume ${\mathbb P}\{||w||\le \lambda \}>0$ for all $\lambda>0$.

\subsection{Affine in the disturbance stochastic MPC}\label{sec:da}

Consider the setting of \cite{goulartandkerrigan2008}, which extends \cite{goulartetal2006} by assuming stochastic disturbances and using an expected value cost.
The control policy is affine in the disturbance and enforces constraint satisfaction robustly (i.e.\ for all disturbance realizations).
%
The state and input are subject to constraints
\begin{equation}\label{goulart_cons}
(x_k,u_k)\in {\bf Z}, \quad k\in {\mathbb N}
\end{equation}
where ${\bf Z}\subseteq {\mathbb R}^{n_x}\times{\mathbb R}^{n_u}$ is a convex and compact set containing the origin in its interior. 
At each discrete time instant $k$ a stochastic optimal control problem is solved. To avoid the computational burden of optimizing over general feedback policies, the predicted control inputs are parameterized as
\[
  u_{i|k} = v_{i|k}+\sum_{j=0}^{i-1}M_{i,j}w_{j|k},\quad i\in\mathbb{N}_{N-1} ,
\]
where the open-loop inputs $v_{i|k}$, $i\in\mathbb{N}_{N-1}$ and the disturbance feedback gains $M_{i,j}$, $j\in\mathbb{N}_{i-1}$, $i\in\{1,\ldots, N-1\}$ are optimization variables computed online, and $u_{i|k}=Kx_{i|k}$ for $i\ge N$, where $A+BK$ is Schur stable.

%

The cost function is given by
\begin{equation}\label{goulart_cost}
  J = {\mathbb E} \Bigl\{x_{N|k}^\top P x_{N|k}+\sum_{i=0}^{N-1} \bigl(
    x_{i|k}^\top Q x_{i|k} + u_{i|k}^\top R u_{i|k} \bigr) \Bigr\},
\end{equation}
where $Q\succeq 0$, $R\succ 0$, $P\succ 0$ and $(A,Q^{1/2})$ is assumed detectable. Although \cite{goulartetal2006,goulartandkerrigan2008} only requires that 
$P$ satisfies a Lyapunov inequality for the terminal dynamics,
here we require that $P$ and $K$ satisfy the algebraic Riccati equation $P=Q+A^\top PA-K^\top(R+B^\top PB)K$, $K=-(R+B^\top PB)^{-1}B^\top PA$. 
A terminal constraint is included in the optimal control problem:
\[
x_N\in{\bf X}_f,
\]
where ${\bf X}_f$ is a positively robust invariant set for system \eqref{lindyn} under the control law $u=Kx$ and constraints \eqref{goulart_cons}. 

The optimal control problem solved at each instant $k$ is:
\begin{equation}\label{oc_goulart}
\begin{aligned}
\underset{\vec{u}_k,\vec{x}_k,\theta_k}{\min}\ \ 
& J \\
\text{subject to}\ \  
& x_{i+1|k}=Ax_{i|k}+Bu_{i|k}+Dw_{i|k}   \\
& u_{i|k}=v_{i|k}+\sum_{j=0}^{i-1}M_{i,j}w_{j|k}
 \\
& (x_{i|k},u_{i|k})\in{\bf Z} \\
& x_{0|k}=x_k, \quad x_{N|k}\in {\bf X}_f\\
& \forall \vec{w}_k \in{\bf W}\times \cdots\times{\bf W} ,\ \forall i\in\mathbb{N}_{N-1}
\end{aligned}
\end{equation}
where $\vec{y}_k$ denotes a predicted sequence $\left\{y_{0|k},\ldots,y_{N-1|k} \right\}$ and $\theta_k=\bigl\{\vec{v}_k,\,M_{i,j},j\in\mathbb{N}_{N-1},\,i\in\{1,\ldots,N-1\}\bigr\}$. 
For polytopic or ellipsoidal ${\bf W}$, this problem is a convex QP or SDP \cite{goulartandkerrigan2008}. The control law is defined for all $x_k\in{\bf X}$, the set of states of (\ref{lindyn}) on which~\eqref{oc_goulart} is feasible, 
as $u_k=v_{0|k}^*(x_k)$, where $(\cdot)^*$ denotes an optimal solution of~(\ref{oc_goulart}). This control input is applied as a receding horizon control law, and the closed-loop system, with state space ${\bf X}\ni x_k$, is given by
\begin{equation}\label{eq:goulart_dynamics}
x_{k+1}=Ax_k+Bv_{0|k}^*(x_k)+Dw_k.
\end{equation}

\begin{assum}\label{ass_final_goulart}
The pair $(A+BK,D)$ is controllable.
\hfill$\blacktriangleleft$
\end{assum}


\begin{prop}\label{asspropgoulart}
Assumptions \ref{distass}, \ref{terminalass}  and \ref{lyapunov}  hold for~(\ref{eq:goulart_dynamics}).
\end{prop}
\begin{proof}
Assumption \ref{distass} holds according to the conditions below \eqref{lindyn}.
Furthermore, ${\bf X}_f\subseteq{\bf X}$ is invariant for \eqref{eq:goulart_dynamics} (since $P,K$ satisfy a Riccati equation, so the MPC law is $u_k=Kx_k$ for all $x_k\in{\bf X}_f$) and bounded (since $\bf Z$ is bounded); also $A+BK$ is necessarily stable and the pair $(A+BK,D)$ is controllable by Assumption~\ref{ass_final_goulart}, which implies that Assumption \ref{terminalass} holds.
Finally, it is shown in \cite{goulartandkerrigan2008} that there is an ISS-Lyapunov function for system \eqref{eq:goulart_dynamics} that satisfies the conditions of Lemma \ref{lemiss}, thus satisying Assumption \ref{lyapunov}.
\end{proof}

We can now give the main convergence results.

\begin{thm}\label{final_goulart}
For any $x_0\in{\bf X}$,
the closed-loop system \eqref{eq:goulart_dynamics} satisfies
$\lim_{k\rightarrow\infty}{\mathbb P}\{ x_k\in \bigoplus_{j=0}^\infty (A+BK)^jD{\bf W}\}=1$
and
\[
\lim_{k\rightarrow\infty}\frac{1}{k}\sum_{j=1}^k
(x_j^{\top\!} Q x_j + u_j^{\!\top\!} R u_j)
\stackrel{a.s.}{=} 
\lim_{k\rightarrow\infty} {\mathbb E}\bigl\{\xi_k^{\top\!}(Q+K^{\!\top\!\!} RK)\xi_k\bigr\}
\]
where $\xi_{k+1} = (A+BK)\xi_k + D w_k$ for all $k\in\mathbb{N}$ and $\xi_0=x_0$.
\end{thm}
\begin{proof}
By Proposition \ref{asspropgoulart}, the MPC formulation considered in this section satisfies the conditions of Theorem~\ref{mainthmgeneral} and Corollary~\ref{maincor}, thus implying the convergence results of the theorem.
\end{proof}

\begin{rem}
  We highlight that this convergence result was demonstrated for the affine in the disturbance policy in \cite{wangetal2008}, which required a redefinition of the cost function. On the other hand, this result is obtained here by including two assumptions: (a)~that the disturbance has a non-zero probability of lying in $\{w : \|w\| \leq \lambda\}$
  for all $\lambda > 0$, and (b)~that $(A+BK,D)$ is controllable. 
\end{rem}

\subsection{Striped affine in the disturbance stochastic MPC}\label{sec:sda}

Consider now the stochastic MPC formulation of \cite{kouvaritakisetal2013}. This also considers a feedback law that is affine in the disturbance, but there are several differences with the parameterization of Section~\ref{sec:da}. First, the control policy includes a fixed gain state feedback law. Second, disturbance feedback gains are computed offline in order to reduce online computational burden. Hence robust constraints are imposed on predicted states and inputs via tightening parameters computed offline. Third, the disturbance feedback has a striped structure that ensures recursive feasibility while extending the disturbance compensation across an infinite horizon.


States and controls are subject to probabilistic constraints
\begin{equation}\label{kouvaritakis_cons}
{\mathbb P}\{f^\top x_{k+1}+ g^\top u_k\le 1 \}\ge p,
\end{equation}
where $g\in {\mathbb R}^{n_x}$, $f\in{\mathbb R}^{ n_u}$, $p\in(0, 1]$. Although \eqref{kouvaritakis_cons} is scalar, multiple constraints of this type can apply simultaneously.

Predicted control inputs have the structure:
\begin{align*}
u_{i|k}&=Kx_{i|k}+c_{i|k}+\sum_{j=1}^{i-1}L_{j}w_{i-j|k}, \quad i\in\mathbb{N}_{N-1} \\
u_{i|k}&=Kx_{i|k}+\sum_{j=1}^{N-1}L_{j}w_{i-j|k}, \quad i=N+l, \ l\in\mathbb{N}
\end{align*}
where $c_{i|k}$, $i\in\mathbb{N}_{N-1}$ are optimization variables and $K$ is a stabilizing feedback gain. The gains $L_{j}$, $j\in\{1,\ldots,N-1\}$ are computed offline so as to minimize a set of constraint tightening parameters that bound the effects of disturbances on constraints
(see~\cite{kouvaritakisetal2013} for details).

The cost function is given by
\begin{equation}\label{kouvaritakis_cost}
J={\mathbb E}\Bigl\{\sum_{i=0}^{\infty} \bigl( x_{i|k}^\top Qx_{i|k}+u_{i|k}^\top Ru_{i|k} -L_{ss}\bigr) \Bigr\},
\end{equation}
where $Q,R\succ 0$ and
$L_{ss}=\lim_{i\to\infty}\! {\mathbb E} ( x_{i|k}^\top\! Qx_{i|k}\!+\! u_{i|k}^\top Ru_{i|k})$ is computed using the terminal mode dynamics and the distribution of the disturbance input.
The gain $K$ satisfies the algebraic Riccati equation, $P=Q+A^\top PA-K^\top(R+B^\top PB)K$, $K=-(R+B^\top PB)^{-1}B^\top PA$. 

The optimal control problem solved at each instant $k$ is:
\begin{equation}\label{oc_kouvaritakis}
  \begin{aligned}
\underset{\vec{u}_k,\vec{x}_k,\vec{c}_k}{\min} \ \ 
& J \\
\text{subject to} \ \
& x_{i+1|k}=Ax_{i|k}+Bu_{i|k}+Dw_{i|k}  \\
& u_{i|k}=c_{i|k}+Kx_{i|k}+\sum_{j=1}^{i-1}L_{j}w_{i-j|k}\\
& {\mathbb P}\{ f^\top x_{i+1|k} + g^\top u_{i|k}\le 1 \} \ge p
\\
& x_{0|k}=x_k \\
& \forall \vec{w}_k \in{\bf W}\times \cdots\times{\bf W} ,\ \forall i\in\mathbb{N}_{N+N_2-1}
\end{aligned}
\end{equation}
where $N_2$ is chosen to be large enough to ensure constraint satisfaction over an infinite prediction horizon.
%
%
The control law is defined by $u_k=c_{0|k}^*(x_k)+Kx_k$ and the procedure is repeated at each sampling instant according to a receding horizon scheme. Thus the closed-loop system is given by
\begin{equation}\label{dynamics_kouvaritakis}
x_{k+1}=(A+BK)x_k+Bc^*_{0|k}(x_k) + Dw_k.
\end{equation}


Since $K$ is the unconstrained optimal feedback gain, we would like to know whether $u_k$ converges to $Kx_k$. A bound of the form $\lim_{k\rightarrow\infty}\mathbb{E}\bigl\{x_{k}^\top Qx_{k}+u_{k}^\top Ru_{k}\bigr\}\le L_{ss}$
is derived in~\cite[Thm.~4.3]{kouvaritakisetal2013} for (\ref{dynamics_kouvaritakis}), where $L_{ss}=l_{ss}+{\mathbb E}\{w^\top P_cw \}$ for some $P_c\succeq 0$, and where $l_{ss}$ is the asymptotic stage cost
${l_{ss}=\lim_{k\rightarrow\infty}{\mathbb E}\left\{x_k^\top(Q+K^\top RK)x_k\right\}}$ for the system~\eqref{lindyn} under $u_k=Kx_k$.
Thus~\cite{kouvaritakisetal2013} provides an asymptotic bound on closed-loop performance, but since $L_{ss} \geq l_{ss}$, this bound does not imply convergence to control law $u_k=Kx_k$.

It is shown in \cite{kouvaritakisetal2013} that if the disturbance is small enough so that, for some $\bar{k}$, $c_{0|\bar{k}}^*(x_{\bar{k}})=0$ and thus $u_k=Kx_k$, then $c_{0|k}^*=0$ for all $k\ge \bar{k}$. This defines a terminal control law within a robustly invariant set ${\bf X}_f$ (namely $\{x \in \mathbb{R}^{n_x}: c_{0|k}^*(x_k)=0\}$), where the closed-loop dynamics are linear.
In order to use the analysis of Section~\ref{sec:results} to show that the state of (\ref{dynamics_kouvaritakis}) converges to ${\bf X}_f\subseteq {\bf X}$, without loss of generality we make the following assumption.


\begin{assum}\label{assboundedstriped}
The state space of~(\ref{dynamics_kouvaritakis}): ${\bf X}=\{x \in \mathbb{R}^{n_x}: \text{(\ref{oc_kouvaritakis}) is feasible}$ $\text{with } x_k=x \}$, is bounded.
\hfill$\blacktriangleleft$
\end{assum}


Note that artificial constraints (chosen so as to always be inactive in practice) may be included in the problem formulation in order to guarantee the boundedness of ${\bf X}$.

We also need some further analysis of the results of \cite{kouvaritakisetal2013}.

\begin{prop}\label{prop3}
(i) The minimization of \eqref{kouvaritakis_cost} is equivalent to the minimization, for suitable matrices $P_x,P_c\succ 0$, of 
\begin{equation}\label{eq_eqcost}
J\defeq x_k^\top P_xx_k+\vec{c}_k^{\, \top} P_c\vec{c}_k.
\end{equation}
(ii) Let $V(x_k)$ be the optimal value of $J$ in problem \eqref{oc_kouvaritakis}, i.e. $V(x_k)\defeq x_k^\top P_xx_k+{\vec{c}}_k^{\, *\top}P_c\vec{c}_k^{\, *}$. 
Then the value function difference $\Delta(x_k,w_k)\defeq V(x_{k+1})-V(x_k)$ satisfies
\begin{equation}\label{lyap_kouvaritakis}
\Delta(x_k,0)\le -\begin{bmatrix}x_k \\ c_{0|k}^*  \end{bmatrix}^\top\begin{bmatrix}Q+K^\top RK & K^\top R \\ RK & R+B^\top P_xB  \end{bmatrix}\begin{bmatrix}x_k \\ c_{0|k}^*  \end{bmatrix}
\end{equation}
(iii) $V(x_k)$ is strictly convex, positive definite and Lipschitz continuous on ${\bf X}$.
\end{prop}
\begin{proof}
(i) The equivalence of minimizing the costs \eqref{kouvaritakis_cost} and \eqref{eq_eqcost} is demonstrated in \cite[Thm.~4.2]{kouvaritakisetal2013} if $P_x,P_c$ satisfy
  \begin{gather}
    \!\!\!\!
    \begin{bmatrix} P_x & 0 \\  0 & P_c \end{bmatrix} - \Psi^\top\begin{bmatrix} P_x & 0 \\ 0 & P_c\end{bmatrix}\Psi  =\begin{bmatrix}\tilde{Q} & K^\top RE \\ E_0^\top RK & E^\top RE   \end{bmatrix} \label{lyapstriped} \\
    \Psi = \begin{bmatrix} A+BK  & BE \\ 0 & M \end{bmatrix} , \nonumber
\end{gather}
where $\tilde{Q}=Q+K^\top RK$ and $E,M$ are such that $E\vec{c}_k =  c_{0|k}$ and  $M\vec{c}_k=
\bigl[c_{1|k}^\top \ \ldots \ c_{N-1|k}^\top \ 0 \bigr]\mbox{}^\top$ for all $\vec{c}_k\in\mathbb{R}^{Nn_u}$.\\
(ii) The bound on the value function difference follows directly from \eqref{dynamics_kouvaritakis} and \eqref{lyapstriped}. \\
(iii) Since $P_x,P_c\succ 0$, \eqref{oc_kouvaritakis} is a strictly convex QP and it follows that $V(\cdot)$ is a continuous, strictly convex piecewise quadratic function~\cite{bemporad02}. Lipschitz continuity of $V(\cdot)$ follows from the boundedness of ${\bf X}$. 
\end{proof}

We use the following result proved in \cite{goulartetal2006} to conclude that $V(x)$ is an ISS-Lyapunov function.

\begin{thm}\label{cor1_kouvaritakis}
The value function $V(x_k)$ is an ISS-Lyapunov function, and the origin is ISS for system \eqref{dynamics_kouvaritakis}.
\end{thm}
\begin{proof}
Since $V(x)$ is positive definite and Lipschitz continuous for $x\in{\bf X}$, it follows that there exist ${\mathcal K}_\infty$-functions $\alpha_1$ and $\alpha_2$ such that \eqref{lyapkappa2} holds.

Let $f(x_k,w_k)=(A+BK)x_k+Bc^*_{k}+Dw_k$. Then \eqref{lyap_kouvaritakis} yields $V(f(x_k,0))-V(x_k)=\Delta(x_k,0) \le -(x_k^\top Qx_k + u^{\top}_kRu_k)$, which implies $V(f(x_k,0))-V(x_k)\le - x_k^\top Qx_k$. Therefore \eqref{isslyap2} holds with $\alpha_3(
\|x\|)=\frac{1}{2}\lambda_{\min}(Q)\|x\|^2$,  where $\lambda_{\min}(Q)$ is the minimum eigenvalue of $Q$. Therefore Lemma~\ref{lemiss2} implies the desired result. \end{proof}

\begin{prop}\label{asspropkouva}
Assumptions \ref{distass}, \ref{terminalass} and \ref{lyapunov} hold for~(\ref{dynamics_kouvaritakis})
\end{prop}
\begin{proof}
Assumption \ref{distass} holds according to the conditions below~\eqref{lindyn}.
Furthermore, it is shown in \cite{kouvaritakisetal2013} that the constraints of \eqref{oc_kouvaritakis} are inactive whenever $x_k$ is sufficiently close to the origin. 
Hence there is a set ${\bf X}_f\subseteq{\bf X}$ such that
$c^\ast_{k}(x_k) = 0$ for $x_k\in{\bf X}_f$, and the recursion of feasibility of \eqref{oc_kouvaritakis} then implies that $c^\ast_{k+1}(x_{k+1}) =0$, so $x_{k+1}\in{\bf X}_f$ and hence ${\bf X}_f$ is invariant for \eqref{eq:goulart_dynamics}. Boundedness of ${\bf X}$ ensures that ${\bf X}_f$ is bounded, which implies that Assumption \ref{terminalass} holds.
Finally, Assumption \ref{lyapunov} holds due to Theorem \ref{cor1_kouvaritakis}.
\end{proof}

This allows us to state the following convergence results.

\begin{thm}
For any $x_0\in{\bf X}$, the closed-loop system \eqref{dynamics_kouvaritakis} satisfies
$\lim_{k\rightarrow\infty}{\mathbb P}\{ x_k\in \bigoplus_{j=0}^\infty (A+BK)^jD{\bf W}\} = 1$ and
\[
\lim_{k\rightarrow\infty}\frac{1}{k}\sum_{j=1}^k
(x_j^{\top\!} Q x_j + u_j^{\!\top\!} R u_j)
\stackrel{a.s.}{=} 
\lim_{k\rightarrow\infty} {\mathbb E}\bigl\{\xi_k^{\top\!}(Q+K^{\!\top\!\!} RK)\xi_k\bigr\}
\]
where $\xi_{k+1} = (A+BK)\xi_k + D w_k$ for all $k\in\mathbb{N}$ and $\xi_0 = x_0$.
\end{thm}
\begin{proof} Identical to the proof of Theorem~\ref{final_goulart}.
\end{proof}

\section{Conclusions}\label{sec:conclusions}

This paper extends and generalizes methods for analysing the convergence of stochastic MPC laws. We provide a set of results applying to general Markov chains to characterize the asymptotic distribution of the system state for the case of control laws that result in linear dynamics within a robustly invariant terminal set of states containing the origin. 
The paper demonstrates that the asymptotic time average of a function of the closed-loop system state is equal to the time average associated with the linear dynamics on the terminal set. These results are obtained using the ISS property of the control law, but the limit directly implied by the ISS Lyapunov inequality would yield a worse bound.
We illustrate the use of the convergence analysis by applying it to two stochastic MPC strategies. Future work will seek to obtain similar results but removing the condition that the pair $(A,D)$ is controllable. 

\bibliographystyle{ieeetr}
\bibliography{diego_library}

\begin{thebibliography}{10}

\bibitem{kouvaritakisandcannon2015}
B.~Kouvaritakis and M.~Cannon, {\em Model Predictive Control: Classical, Robust
  and Stochastic}.
\newblock Advanced Textbooks in Control and Signal Processing, Springer, 2015.

\bibitem{mesbah2016}
A.~Mesbah, ``Stochastic model predictive control: An overview and perspectives
  for future research,'' {\em IEEE Control Systems}, vol.~36, pp.~30--44, Dec
  2016.

\bibitem{mayne2014}
D.~Q. Mayne, ``Model predictive control: Recent developments and future
  promise,'' {\em Automatica}, vol.~50, no.~12, pp.~2967 -- 2986, 2014.

\bibitem{cannonetal2009a}
M.~Cannon, B.~Kouvaritakis, and X.~Wu, ``Probabilistic constrained {MPC} for
  multiplicative and additive stochastic uncertainty,'' {\em IEEE Trans.\
  Automatic Control}, vol.~54, no.~7, pp.~1626--1632, 2009.

\bibitem{cannonetal2009b}
M.~Cannon, B.~Kouvaritakis, and D.~Ng, ``Probabilistic tubes in linear
  stochastic model predictive control,'' {\em Systems \& Control Letters},
  vol.~58, no.~10, pp.~747--753, 2009.

\bibitem{kouvaritakisetal2010}
B.~Kouvaritakis, M.~Cannon, S.~V. Rakovi{\'c}, and Q.~Cheng, ``Explicit use of
  probabilistic distributions in linear predictive control,'' {\em Automatica},
  vol.~46, no.~10, pp.~1719--1724, 2010.

\bibitem{chatterjeeandlygeros2015}
D.~Chatterjee and J.~Lygeros, ``On stability and performance of stochastic
  predictive control techniques,'' {\em IEEE Transactions on Automatic
  Control}, vol.~60, pp.~509--514, Feb 2015.

\bibitem{paulsonetal2015}
J.~A. Paulson, S.~Streif, and A.~Mesbah, ``Stability for receding-horizon
  stochastic model predictive control,'' in {\em 2015 American Control
  Conference (ACC)}, pp.~937--943, July 2015.

\bibitem{mishraetal2016}
P.~K. Mishra, D.~E. Quevedo, and D.~Chatterjee, ``Dropout feedback parametrized
  policies for stochastic predictive controller,'' {\em IFAC-PapersOnLine},
  vol.~49, no.~18, pp.~59 -- 64, 2016.
\newblock 10th IFAC Symposium on Nonlinear Control Systems NOLCOS 2016.

\bibitem{lorenzenetal2017}
M.~Lorenzen, F.~Dabbene, R.~Tempo, and F.~Allgower, ``Constraint-tightening and
  stability in stochastic model predictive control,'' {\em IEEE Transactions on
  Automatic Control}, 2016.

\bibitem{goulartandkerrigan2008}
P.~J. Goulart and E.~C. Kerrigan, ``Input-to-state stability of robust receding
  horizon control with an expected value cost,'' {\em Automatica}, vol.~44,
  no.~4, pp.~1171--1174, 2008.

\bibitem{kouvaritakisetal2013}
B.~Kouvaritakis, M.~Cannon, and D.~Mu{\~n}oz-Carpintero, ``Efficient prediction
  strategies for disturbance compensation in stochastic {MPC},'' {\em
  International Journal of Systems Science}, vol.~44, no.~7, pp.~1344--1353,
  2013.

\bibitem{meynandtweedie2009}
S.~Meyn, R.~Tweedie, and P.~Glynn, {\em Markov Chains and Stochastic
  Stability}.
\newblock Cambridge Mathematical Library, CUP, 2009.

\bibitem{jiangandwang2001}
Z.-P. Jiang and Y.~Wang, ``Input-to-state stability for discrete-time nonlinear
  systems,'' {\em Automatica}, vol.~37, no.~6, pp.~857 -- 869, 2001.

\bibitem{goulartetal2006}
P.~J. Goulart, E.~C. Kerrigan, and J.~M. Maciejowski, ``Optimization over state
  feedback policies for robust control with constraints,'' {\em Automatica},
  vol.~42, no.~4, pp.~523--533, 2006.

\bibitem{wangetal2008}
C.~Wang, C.~J. Ong, and M.~Sim, ``Constrained linear system under disturbance
  feedback: Convergence with probability one,'' in {\em 2008 47th IEEE
  Conference on Decision and Control}, pp.~2820--2825, Dec 2008.

\bibitem{bemporad02}
A.~Bemporad, M.~Morari, V.~Dua, and E.~Pistikopoulos, ``The explicit linear
  quadratic regulator for constrained systems,'' {\em Automatica}, vol.~38,
  no.~1, pp.~3--20, 2002.

\end{thebibliography}

\end{document}